\documentclass[onefignum,onetabnum]{siamart190516}

\usepackage{lipsum}
\usepackage{amsfonts}
\usepackage{graphicx}
\usepackage{epstopdf}
\usepackage{algorithmic}

\usepackage{graphicx}
\usepackage{epstopdf}
\usepackage{lineno,hyperref}
 \usepackage{geometry}
\usepackage{epsfig}
\usepackage{lineno,hyperref}
\usepackage{bookmark}
\usepackage{multirow}
 \usepackage{amsfonts}
\usepackage{multirow}
 \usepackage{algorithm}
 \usepackage[figuresright]{rotating}
 \usepackage{amscd}
\usepackage{mathrsfs}
\usepackage{booktabs,longtable}
\usepackage{indentfirst}
\usepackage{subfigure}
\usepackage{amstext}
\usepackage{amssymb}
\usepackage{fancyhdr}
\usepackage{amsmath}
\usepackage{cleveref}

\ifpdf
  \DeclareGraphicsExtensions{.eps,.pdf,.png,.jpg}
\else
  \DeclareGraphicsExtensions{.eps}
\fi

\newsiamremark{remark}{Remark}
\newsiamremark{hypothesis}{Hypothesis}
\crefname{hypothesis}{Hypothesis}{Hypotheses}
\newsiamthm{claim}{Claim}

\headers{A Novel Greedy Kaczmarz Method }{Hanyu Li and Yanjun Zhang}

\title{A Novel Greedy Kaczmarz Method For Solving Consistent Linear Systems\thanks{Submitted to arXiv.
\funding{This work was funded by the National Natural Science Foundation of China (No. 11671060) and the Natural Science Foundation Project of CQ CSTC (No. cstc2019jcyj-msxmX0267).}}}

\author{Hanyu Li\thanks{College of Mathematics and Statistics, Chongqing University, Chongqing 401331, P.R. China
  (\email{lihy.hy@gmail.com or hyli@cqu.edu.cn;\ yjzhang@cqu.edu.cn}).}
\and Yanjun Zhang
\footnotemark[2]}

\usepackage{amsopn}

\ifpdf
\hypersetup{
  pdftitle={ A Novel Greedy Kaczmarz Method For Solving Consistent Linear Systems},
  pdfauthor={Hanyu Li,\ Yanjun Zhang}
}
\fi

\begin{document}

\maketitle

\begin{abstract}
 With a quite different way to determine the working rows, we propose a novel greedy Kaczmarz method for solving consistent linear systems. Convergence analysis of the new method is provided. Numerical experiments show that, for the same accuracy, our method outperforms the  greedy randomized Kaczmarz method and the relaxed greedy randomized Kaczmarz method introduced recently by Bai and Wu [Z.Z. BAI AND W.T. WU, On greedy randomized Kaczmarz method for solving large sparse linear systems, SIAM J. Sci. Comput., 40 (2018), pp. A592--A606; Z.Z. BAI AND W.T. WU, On relaxed greedy randomized Kaczmarz methods for solving large sparse linear systems, Appl. Math. Lett., 83 (2018), pp. 21--26]
in term of the computing time.
\end{abstract}



\section{Introduction}
We consider the following consistent linear systems
\begin{equation}
\label{1}
Ax=b,
\end{equation}
where $A\in C^{m\times n}$, $b\in C^{m}$, and $x$ is the $n$-dimensional unknown vector. As we know, the Kaczmarz method \cite{kaczmarz1} is a popular so-called row-action method for solving the systems (\ref{1}). In 2009, Strohmer and Vershynin \cite{Strohmer2009} proved the linear convergence of the randomized Kaczmarz (RK) method. Following that, Needell \cite{Needell2010} found that the RK method is not converge to the ordinary least squares solution when the system is inconsistent. To overcome it, Zouzias and Freris \cite{Completion2013} extended the RK method to the randomized extended Kaczmarz (REK) method. Later, Ma, Needell, and Ramdas \cite{Completion2015} provided a unified theory of these related iterative methods in all possible system settings. Recently, many works on Kaczmarz methods were reported; see for example \cite{Bai2018, Bai2018r, Gao2019, Dukui2019, Wu2020, Chen2020, Niu2020} and references therein.

In 2018, Bai and Wu \cite{Bai2018} first constructed a greedy randomized Kaczmarz (GRK) method by introducing an efficient probability criterion for selecting the working rows from the coefficient matrix $A$, which avoids a weakness of the one adopted in the RK method. As a result, the GRK method is faster than the RK method in terms of the number of iterations and computing time. Subsequently, based on the GRK method, a so-called relaxed greedy randomized Kaczmarz (RGRK) method was proposed in \cite{Bai2018r} by introducing a relaxation parameter $\theta$, which makes the convergence factor of the RGRK method be smaller than that of the GRK method when it is in $[\frac{1}{2}, 1]$, and the convergence factor reaches the minimum  when $\theta=1$. For the latter case, i.e., $\theta=1$, Du and Gao \cite{Gao2019} called it the maximal weighted residual Kaczmarz method and carried out extensive experiments to test this method. By the way, the idea of greed applied in \cite{Bai2018,Bai2018r} also has wide applications, see for example \cite{Griebel2012,Nguyen2017,Liu2019,Bai2019} and references therein.

In the present paper, we propose a novel greedy Kaczmarz (GK) method. Unlike the GRK and RGRK methods, the new method adopts a quite different way to determine the working rows of the matrix $A$ and hence needs less computing time in each iteration; see the detailed analysis before Algorithm \ref{alg3} below.  Consequently, the GK method can outperform the  GRK and RGRK methods in term of the computing time. This result is confirmed by extensive numerical experiments,  which show that, for the same accuracy, the GK method requires almost the same number of iterations as those of the GRK and RGRK methods, but spends less computing time. In addition, we also prove the convergence of the GK method in theory.

The rest of this paper is organized as follows. In Section \ref{sec2}, notation and some preliminaries are provided. We present our novel GK method and its convergence properties in Section \ref{sec3}. Experimental results are given in Section \ref{sec4}.

\section{Notation and preliminaries }\label{sec2}
For a vector $z\in C^{n}$, $z^{(j)}$ represents its $j$th entry. For a matrix $G=(g_{ij})\in C^{m\times n}$, $G^{(i)}$, $\|G\|_2$, and $\|G\|_F$ denote its $i$th row, spectral norm, and Frobenius norm, respectively. In addition, we denote 
the smallest positive eigenvalue of $G^{\ast}G$ by $\lambda_{\min}\left( G^{\ast} G\right)$, where $(\cdot)^{\ast}$ denotes the conjugate transpose of a vector or a matrix, and the number of elements of a set $\mathcal{W}$ by $|\mathcal{W}|$.

In what follows, we use $x_{\star}=A^{\dag}b$, with $A^{\dag}$ being the Moore-Penrose pseudoinverse, to denote the least-Euclidean-norm solution to the systems (\ref{1}). For finding this solution, 
Bai and Wu \cite{Bai2018} proposed the GRK method listed as follows, where $r_k = b-Ax_k$ denotes the residual vector.
\begin{algorithm}
\caption{The GRK method}
\label{alg1}
\begin{algorithmic}
\STATE{\mbox{Input:} ~$A\in C^{m\times n}$, $b\in C^{m}$, $\ell$ , initial estimate $x_0$}
\STATE{\mbox{Output:} ~$x_\ell$}
\FOR{$k=0, 1, 2, \ldots, \ell-1$}
\STATE{Compute
\begin{align}
  \epsilon_{k}=\frac{1}{2}\left(\frac{1}{\left\|r_{k}\right\|_{2}^{2}} \max _{1 \leq i_{k} \leq m}\left\{\frac{\left|r^{\left(i_{k}\right)}_k\right|^{2}}{\left\|A^{\left(i_{k}\right)}\right\|_{2}^{2}}\right\}+\frac{1}{\|A\|_{F}^{2}}\right).\notag
\end{align}}
\STATE{Determine the index set of positive integers
\begin{align}
 \mathcal{U}_{k}=\left\{i_{k}\Bigg|| r^{\left(i_{k}\right)}_k|^{2} \geq \epsilon_{k}\left\|r_{k}\right\|_{2}^{2}\|A^{\left(i_{k}\right)}\|_{2}^{2}\right\}.\notag
\end{align}}
\STATE{Compute the $i$th entry $\tilde{r}_k^{(i)}$ of the vector $\tilde{r}_k$ according to
  $$
\tilde{r}_{k}^{(i)}=\left\{\begin{array}{ll}{r^{(i)}_{k},} & {\text { if } i \in \mathcal{U}_{k}}, \\ {0,} & {\text { otherwise. }}\end{array}\right.
$$}
\STATE{Select $i_{k} \in \mathcal{U}_{k}$ with probability $\operatorname{Pr}\left(\mathrm{row}=i_{k}\right)=\frac{|\tilde{r}_{k}^{\left(i_{k}\right)}|^{2}}{\left\|\tilde{r}_{k}\right\|_{2}^{2}}$.}
\STATE{Set
$$ x_{k+1}=x_{k}+\frac{   r^{ (i_{k}) }_{k} }{\| A^{(i_{k})}\|_{2}^{2}}( A^{(i_{k})})^{\ast}.$$}
\ENDFOR
\end{algorithmic}
\end{algorithm}

From the definitions of  $\epsilon_{k}$ and $\mathcal{U}_{k}$ in Algorithm \ref{alg1}, we have  that if $\ell\in \mathcal{U}_{k}$, then
\begin{eqnarray*}
\frac{| r^{\left(\ell\right)}_{k}|^{2}}{\|A^{\left(\ell\right)}\|_{2}^{2}}\geq\frac{1}{2}\left( \max \limits _{1 \leq i_{k}  \leq m}\left\{\frac{\left|r^{\left(i_{k}\right)}_{k}\right|^{2}}{\left\|A^{\left(i_{k}\right)}\right\|_{2}^{2}}\right\}+\frac{\left\|r_{k}\right\|_{2}^{2}}{\|A\|_{F}^{2}}\right).
\end{eqnarray*}
Note that
\begin{equation*}
\label{3}
\max \limits _{1 \leq i_{k}  \leq m}\left\{\frac{\left|r^{\left(i_{k}\right)}_{k}\right|^{2}}{\left\|A^{\left(i_{k}\right)}\right\|_{2}^{2}}\right\}\geq\sum_{i_k=1}^{m}\frac{\|A^{(i_k)}\|^2_2}{\| A\|^2_F}\frac{ \left|r^{\left(i_{k}\right)}_{k}\right|^{2}}{\|A^{(i_k)}\|^2_2} \notag= \frac{\left\|r_{k}\right\|^{2}_2}{\|A\|_{F}^{2}}.
\end{equation*}
Thus, we can't conclude that if $\ell\in \mathcal{U}_{k}$, then
\begin{eqnarray*}
\frac{| r^{\left(\ell\right)}_{k}|^{2}}{\|A^{\left(\ell\right)}\|_{2}^{2}}\geq \max \limits _{1 \leq i_k  \leq m}\left\{\frac{| r^{\left(i_{k}\right)}_{k}|^{2}}{\|A^{\left(i_{k}\right)}\|_{2}^{2}}\right\}, \textrm{ i.e., } \ \frac{| r^{\left(\ell\right)}_{k}|^{2}}{\|A^{\left(\ell\right)}\|_{2}^{2}}= \max \limits _{1 \leq i_k  \leq m}\left\{\frac{| r^{\left(i_{k}\right)}_{k}|^{2}}{\|A^{\left(i_{k}\right)}\|_{2}^{2}}\right\}.
\end{eqnarray*}
As a result, there may exist some $\ell\in \mathcal{U}_{k}$ such that
\begin{equation}
\label{3343434}
 \frac{| r^{\left(\ell\right)}_{k}|^{2}}{\|A^{\left(\ell\right)}\|_{2}^{2}}<\max \limits _{1 \leq i_k  \leq m}\left\{\frac{| r^{\left(i_{k}\right)}_{k}|^{2}}{\|A^{\left(i_{k}\right)}\|_{2}^{2}}\right\}.
\end{equation}
Meanwhile, from the update formula, for any $i_k\in \mathcal{U}_{k}$, we have
\begin{equation}
\label{2}
\|x_{k+1}-x_{k}\|^2_2=\frac{| r^{\left(i_k\right)}_{k}|^{2}}{\|A^{\left(i_k\right)}\|_{2}^{2}}.
\end{equation}
Thus, combining (\ref{3343434}) and (\ref{2}), we can find that we can't make sure any row with the index from the index set $\mathcal{U}_{k}$ make the distance between $x_{k+1}$ and $x_{k}$ be the largest when finding $x_{k+1}$. Moreover, to compute $\epsilon_{k}$, we have to compute the norm of each row of the matrix $A$.

Based on the GRK method, 
Bai and Wu \cite{Bai2018r} further designed  the RGRK method by introducing a relaxation parameter, 
which is listed in 
Algorithm \ref{alg2}.
\begin{algorithm}
\caption{The RGRK method}
\label{alg2}
\begin{algorithmic}
\STATE{\mbox{Input:} ~$A\in C^{m\times n}$, $b\in C^{m}$, $\theta\in [0,1]$, $\ell$ , initial estimate $x_0$}
\STATE{\mbox{Output:} ~$x_\ell$}
\FOR{$k=0, 1, 2, \ldots, \ell-1$}
\STATE{Compute
\begin{align}
  \varepsilon_{k}=\frac{\theta}{\left\|r_{k}\right\|_{2}^{2}} \max _{1 \leq i_{k} \leq m}\left\{\frac{\left|r^{\left(i_{k}\right)}_k\right|^{2}}{\left\|A^{\left(i_{k}\right)}\right\|_{2}^{2}}\right\}+\frac{1-\theta}{\|A\|_{F}^{2}}.\notag
\end{align}}
\STATE{Determine the index set of positive integers
\begin{align}
 \mathcal{V}_{k}=\left\{i_{k}\Bigg|| r^{\left(i_{k}\right)}_k|^{2} \geq \varepsilon_{k}\left\|r_{k}\right\|_{2}^{2}\|A^{\left(i_{k}\right)}\|_{2}^{2}\right\}.\notag
\end{align}}
\STATE{Compute the $i$th entry $\tilde{r}_k^{(i)}$ of the vector $\tilde{r}_k$ according to
  $$
\tilde{r}_{k}^{(i)}=\left\{\begin{array}{ll}{r^{(i)}_{k},} & {\text { if } i \in \mathcal{V}_{k}}, \\ {0,} & {\text { otherwise. }}\end{array}\right.
$$}
\STATE{Select $i_{k} \in \mathcal{V}_{k}$ with probability $\operatorname{Pr}\left(\mathrm{row}=i_{k}\right)=\frac{|\tilde{r}_{k}^{\left(i_{k}\right)}|^{2}}{\left\|\tilde{r}_{k}\right\|_{2}^{2}}$.}
\STATE{Set
$$ x_{k+1}=x_{k}+\frac{   r^{ (i_{k}) }_{k} }{\| A^{(i_{k})}\|_{2}^{2}}( A^{(i_{k})})^{\ast}.$$}
\ENDFOR
\end{algorithmic}
\end{algorithm}

It is easy to see that when $\theta=\frac{1}{2} $, the RGRK method is just the GRK method. Bai and Wu \cite{Bai2018r} showed that the convergence factor of the RGRK method is smaller than that of the GRK method when 
$\theta\in[\frac{1}{2}, 1]$, and the convergence factor reaches the minimum when $\theta=1$. For the latter case, we have that if $\ell\in \mathcal{V}_{k}$, then
\begin{eqnarray*}
\frac{| r^{\left(\ell\right)}_{k}|^{2}}{\|A^{\left(\ell\right)}\|_{2}^{2}}= \max \limits _{1 \leq i_k  \leq m}\left\{\frac{| r^{\left(i_{k}\right)}_{k}|^{2}}{\|A^{\left(i_{k}\right)}\|_{2}^{2}}\right\}.
\end{eqnarray*}
From the analysis following the Algorithm \ref{alg2}, in this case, the row with the index from the index set $\mathcal{V}_{k}$ can make the distance between $x_{k+1}$ and $x_{k}$ be the largest for any possible $x_{k+1}$. However, we still needs to compute the norm of each row of the matrix $A$ when computing $\varepsilon_{k}$.


\section{A novel greedy Kaczmarz method}\label{sec3}

On basis of the analysis of the GRK and RGRK methods and inspired by some recent works on selection strategy for working index based on the maximum residual \cite{Nutini2018,Haddock2019,Rebrova2019}, we design a new method for solving consistent linear systems which includes two main steps. In the first step, we use the maximum entries of the residual vector $r_k$ to determine an index set $\mathcal{R}_{k}$ whose specific definition is given in Algorithm \ref{alg3}. In the second step, we capture an index from the set $\mathcal{R}_{k}$ with which we can make sure the distance between $x_{k+1}$ and $x_{k}$ be the largest for any possible $x_{k+1}$. On a high level, the new method seems to change the order of the two main steps of Algorithm \ref{alg1} or Algorithm \ref{alg2}. However, comparing with the GRK and RGRK methods, we do not need to calculate the norm of each row of the matrix $A$ any longer in Algorithm \ref{alg3}, and, like the RGRK method with $\theta=1$, our method always makes the distance between $x_{k+1}$ and $x_{k}$ be the largest when finding $x_{k+1}$. 
In fact, the new method combines the maximum residual rule and the maximum distance rule. 
These characters make the method reduce the computation cost at each iteration and hence behaves better in the computing time, which is confirmed by numerical experiments given in Section \ref{sec4}.

Based on the above introduction, we propose the following algorithm, i.e., Algorithm \ref{alg3}.
\begin{algorithm}
\caption{The GK method}
\label{alg3}
\begin{algorithmic}
\STATE{\mbox{Input:} ~$A\in C^{m\times n}$, $b\in C^{m}$, $\ell$ , initial estimate $x_0$}
\STATE{\mbox{Output:} ~$x_\ell$}
\FOR{$k=0, 1, 2, \ldots, \ell-1$}
\STATE{Determine the index set of positive integers
 $$\mathcal{R}_{k}=\left\{\tilde{i}_{k}\Bigg| \tilde{i}_{k}= {\rm arg} \max \limits _{1 \leq i  \leq m}\left|r^{(i)}_k\right|\right\}.$$}
\STATE{Compute
$$i_{k}={\rm arg} \max \limits _{\tilde{i}_{k}\in \mathcal{R}_{k}}\left\{\frac{\left|r^{(\tilde{i}_{k})}_k\right|^2 }{\left\|A^{(\tilde{i}_{k})}\right\|^2_{2}  }\right\}.$$}
\STATE{Set
$$x_{k+1}=x_{k}+\frac{ r^{(i_{k})}_k}{ \| A^{\left(i_{k}\right)} \|_{2}^{2}}( A^{(i_{k})})^{\ast}.$$}
\ENDFOR
\end{algorithmic}
\end{algorithm}

\begin{remark}
\label{rmk}
Note that if
$$\left|r^{(i_k)}_k\right|=  \max \limits _{1 \leq i  \leq m}\left|r^{(i)}_k\right|,$$
then $i_k\in\mathcal{R}_{k}.$ So the index set $\mathcal{R}_{k}$ in Algorithm \ref{alg3} is nonempty for all iteration index $k$.
\end{remark}

\begin{remark}
\label{rmk110}
Like Algorithm \ref{alg1} or Algorithm \ref{alg2}, we can use the values of $\frac{\left|r^{(\tilde{i}_{k})}_k\right|^2 }{\left\|A^{(\tilde{i}_{k})}\right\|^2_{2}  }$ for $\tilde{i}_{k}\in \mathcal{R}_{k}$ as a probability selection criterion to devise a randomized version of Algorithm \ref{alg3}. In this case, the convergence factor may be a little worse than that of Algorithm \ref{alg3} because, for the latter, the index is selected based on the largest value of $\frac{\left|r^{(\tilde{i}_{k})}_k\right|^2 }{\left\|A^{(\tilde{i}_{k})}\right\|^2_{2}  }$ for $\tilde{i}_{k}\in \mathcal{R}_{k}$, which make the distance between $x_{k+1}$ and $x_{k}$ be the largest for any possible $x_{k+1}$.
\end{remark}

\begin{remark}
\label{rmk3}
 To the best of our knowledge, the idea of Algorithm \ref{alg3} is brand new in the fields of designing greedy Kaczmarz type algorithms and we don't find it in any work on greedy Gauss-Seidel methods either. So it is interesting to apply this idea to Gauss-Seidel methods to devise some new greedy Gauss-Seidel algorithms for solving other problems like large linear least squares problems. We will consider this topic in a subsequent paper.
\end{remark}

In the following, we give the convergence theorem of the GK method.

\begin{theorem}
\label{theorem1}
The iteration sequence $\{x_k\}_{k=0}^\infty$ generated by Algorithm \ref{alg3}, starting from an initial guess $x_0\in C^{n}$ in the column space of $A^{*}$, converges linearly to the least-Euclidean-norm solution $x_{\star}=A^{\dag}b$ and
\begin{equation}
\label{4}
  \| x_1-x_\star\|^2_{2} \leq \left(1-\frac{1}{|\mathcal{R}_{0}|}\cdot\frac{1}{\sum\limits_{i_0\in \mathcal{R}_{0}}\|A^{(i_0)}\|^2_2 }\cdot\frac{1}{m} \cdot\lambda_{\min}\left( A^{*} A\right)\right)\cdot\| x_0-x_\star\|^2_{2},
\end{equation}
and
 \begin{equation}
\label{5}
\| x_{k+1}-x_\star\|^2_{2} \leq\left(1-\frac{1}{|\mathcal{R}_{k}|}\cdot\frac{1}{\sum\limits_{i_k\in \mathcal{R}_{k}}\|A^{(i_k)}\|^2_2 }\cdot\frac{1}{m-1} \cdot\lambda_{\min}\left( A^{*} A\right) \right)\| x_k-x_\star\|^2_{2},~~k=1, 2, \ldots .
\end{equation}
Moreover, let $\alpha=\max \{|\mathcal{R}_{k}|\}$, $\beta=\max\{\sum\limits_{i_k\in \mathcal{R}_{k}}\|A^{(i_k)}\|^2_2\}, k=0, 1, 2, \ldots.$ Then,
\begin{equation}
\label{6}
\| x_{k}-x_\star\|^2_{2} \leq\left(1-\frac{\lambda_{\min}\left( A^{*} A\right) }{\alpha\cdot \beta\cdot(m-1)} \right)^{k-1}\left(1- \frac{\lambda_{\min}\left( A^{*} A\right)}{|\mathcal{R}_{0}| \cdot\sum\limits_{i_0\in \mathcal{R}_{0}}\|A^{(i_0)}\|^2_2 \cdot m  }\right)\cdot\| x_0-x_\star\|^2_{2},~~k=1, 2, \ldots .
\end{equation}
\end{theorem}
\begin{proof}

From the update formula in Algorithm \ref{alg3}, we have
 $$x_{k+1}-x_{k}=\frac{ r^{(i_{k})}_k}{ \| A^{\left(i_{k}\right)} \|_{2}^{2}}( A^{(i_{k})})^{\ast},$$
which implies that $x_{k+1}-x_{k}$ is parallel to $( A^{(i_{k})})^{\ast}$. Meanwhile,
\begin{eqnarray*}
A^{\left(i_{k}\right)}(x_{k+1}-x_{\star})&=&A^{\left(i_{k}\right)}\left(x_{k}-x_\star+\frac{ r^{(i_{k})}_k}{ \| A^{\left(i_{k}\right)} \|_{2}^{2}}( A^{(i_{k})})^{\ast}\right)
\\
&=& A^{\left(i_{k}\right)}\left(x_{k}-x_\star\right)+r^{(i_{k})}_k,
\end{eqnarray*}
which together with the fact $Ax_{\star}=b$ gives
\begin{eqnarray*}
A^{\left(i_{k}\right)}(x_{k+1}-x_{\star})&=& ( A^{\left(i_{k}\right)}x_{k}-b^{\left(i_{k}\right)})+(b^{\left(i_{k}\right)}- A^{\left(i_{k}\right)}x_{k})=0.
\end{eqnarray*}
Then $x_{k+1}-x_{\star}$ is orthogonal to $A^{\left(i_{k}\right)}$. Thus, the vector $x_{k+1}-x_{k}$ is perpendicular to the vector $x_{k+1}-x_{\star}$. By the Pythagorean theorem, we get
\begin{equation}
\label{7}
\|x_{k+1}-x_{\star}\|^2_2=\|x_{k}-x_{\star}\|^2_2-\|x_{k+1}-x_{k}\|^2_2.
\end{equation}
On the other hand, from Algorithm \ref{alg3}, we have
$$\left|r^{(i_{k})}_k\right|=\max\limits_{1\leq i\leq m}  \left|r^{(i)}_k\right|
~{\rm and}~
\frac{ \left|r^{(i_{k})}_k\right|^2 }{ \left\| A^{\left(i_{k}\right)}\right \|_{2}^{2}}=\max\limits_{i\in \mathcal{R}_{k}}\frac{ \left|r^{(i)}_k\right|^2 }{ \left\| A^{\left(i\right)}\right \|_{2}^{2}}.
$$
Then
\begin{align}
\left\|x_{k+1}-x_{k}\right\|^2_2&=~\frac{ \left|r^{(i_{k})}_k\right|^2 }{ \left\| A^{\left(i_{k}\right)}\right \|_{2}^{2}} \geq~ \sum\limits_{i_k\in \mathcal{R}_{k}} \frac{\frac{ \left|r^{(i_{k})}_k\right|^2 }{ \left\| A^{\left(i_{k}\right)}\right \|_{2}^{2}}}{\sum\limits_{i\in \mathcal{R}_{k}}\frac{ \left|r^{(i)}_k\right|^2 }{ \left\| A^{\left(i\right)}\right \|_{2}^{2}} }\cdot \frac{ \left|r^{(i_{k})}_k\right|^2 }{ \left\| A^{\left(i_{k}\right)}\right \|_{2}^{2}} \notag\\
&\geq ~\sum\limits_{i_k\in \mathcal{R}_{k}} \frac{1}{|\mathcal{R}_{k}|}\cdot\frac{ \left|r^{(i_{k})}_k\right|^2 }{ \left\| A^{\left(i_{k}\right)}\right \|_{2}^{2}} = ~\sum\limits_{i_k\in \mathcal{R}_{k}} \frac{1}{|\mathcal{R}_{k}|}\cdot\frac{ \max\limits_{1\leq i\leq m}  \left|r^{(i)}_k\right|^2}{ \left\| A^{\left(i_{k}\right)}\right \|_{2}^{2}}.\label{8}
\end{align}
Thus, substituting (\ref{8}) into (\ref{7}), we obtain
\begin{equation}
\label{71212}
\left\|x_{k+1}-x_{\star}\right\|^2_{2}\leq\left\|x_{k}-x_{\star}\right\|^2_{2}-~\sum\limits_{i_k\in \mathcal{R}_{k}} \frac{1}{|\mathcal{R}_{k}|}\cdot\frac{ \max\limits_{1\leq i\leq m}  \left|r^{(i)}_k\right|^2}{ \left\| A^{\left(i_{k}\right)}\right \|_{2}^{2}} .
\end{equation}

For $k=0$, we have
\begin{align*}
\max\limits_{1\leq i\leq m}  \left|r^{(i)}_0\right|^2&= ~\max\limits_{1\leq i\leq m}  \left|r^{(i)}_0\right|^2\cdot \frac{\left\|r_0\right\|^2_2}{\sum\limits_{i=1}^{m}\left|r^{(i)}_0\right|^2}\geq~\frac{1}{m}\cdot\left\| r_0\right\|^2_2,
\end{align*}
which together with a result from \cite{Bai2018}:
\begin{align}
\|Ax\|^2_2\geq\lambda_{\min}\left( A^{*} A\right)\|x\|^2_2  \label{1210}
\end{align}
is valid for any vector $x$ in the column space of $A^{*}$, implies
\begin{align}
\max\limits_{1\leq i\leq m}\left|r^{(i)}_0\right|^2&\geq~\frac{1}{m}\cdot\lambda_{\min}\left( A^{*} A\right) \cdot\left\| x_\star-x_0\right\|^2_2. \label{10}
\end{align}
Thus, substituting (\ref{10}) into (\ref{71212}), we obtain
\begin{align}
\left\|x_{1}-x_{\star}\right\|^2_{2}&\leq~\left\|x_{0}-x_{\star}\right\|^2_{2}-\sum\limits_{i_0\in \mathcal{R}_{0}} \frac{1}{|\mathcal{R}_{0}|}\cdot\frac{ 1}{ \left\| A^{\left(i_{0}\right)}\right \|_{2}^{2}} \cdot \frac{1}{m}\cdot\lambda_{\min}\left( A^{*} A\right) \cdot\left\|x_0-x_\star\right\|^2_{2} \notag\\
&=~\left(1-\frac{1}{|\mathcal{R}_{0}|}\cdot\frac{1}{\sum\limits_{i_0\in \mathcal{R}_{0}}\|A^{(i_0)}\|^2_2 }\cdot\frac{1}{m} \cdot\lambda_{\min}\left( A^{*} A\right)\right)\cdot\| x_0-x_\star\|^2_{2}, \notag
\end{align}
which is just the estimate (\ref{4}).

%
For $k\geq1$, we have
\begin{align*}
\max\limits_{1\leq i\leq m}\left|r^{(i)}_k\right|^2&= ~\max\limits_{1\leq i\leq m}\left|r^{(i)}_k\right|^2 \cdot \frac{\left\| r_k\right\|^2_2}{\sum\limits_{i=1}^{m}\left|r^{(i)}_k\right|^2}.
\end{align*}
Note that, according to the update formula in Algorithm \ref{alg3}, it is easy to obtain
\begin{align} r_{k}^{\left(i_{k-1}\right)} &=b^{\left(i_{k-1}\right)}-A^{\left(i_{k-1}\right)} x_{k} \notag\\ &=b^{\left(i_{k-1}\right)}-A^{\left(i_{k-1}\right)}\left(x_{k-1}+\frac{r^{\left(i_{k-1}\right)}_{k-1}}{\left\|A^{\left(i_{k-1}\right)}\right\|_{2}^{2}}\left(A^{\left(i_{k-1}\right)}\right)^{*}\right) \notag\\ &=b^{\left(i_{k-1}\right)}-A^{\left(i_{k-1}\right)} x_{k-1}-r^{\left(i_{k-1}\right)}_{k-1} \notag\\
 &=0.\label{9}
 \end{align}
Then
\begin{align*}
\max\limits_{1\leq i\leq m}\left|r^{(i)}_k\right|^2&=~\max\limits_{1\leq i\leq m}\left|r^{(i)}_k\right|^2 \cdot \frac{\left\| r_k\right\|^2_2}{\sum\limits _{i=1 \atop i \neq i_{k-1}}^{m}\left|r^{(i)}_k\right|^2}\geq~ \frac{1}{m-1}\cdot\left\|r_k\right\|^2_2,
\end{align*}
which together with  (\ref{1210}) yields
\begin{align}
\max\limits_{1\leq i\leq m}\left|r^{(i)}_k\right|^2&\geq~ \frac{1}{m-1}\cdot\lambda_{\min}\left( A^{*} A\right)\left\|x_\star-x_k\right\|^2_2.\label{12}
\end{align}
Thus, substituting (\ref{12}) into (\ref{71212}), we get
\begin{align}
\left\|x_{k+1}-x_{\star}\right\|^2_{2}&\leq~\left\|x_{k}-x_{\star}\right\|^2_{2}-\sum\limits_{i_k\in \mathcal{R}_{k}} \frac{1}{|\mathcal{R}_{k}|}\cdot  \frac{ 1}{ \left\| A^{\left(i_{k}\right)}\right \|_{2}^{2}}\cdot \frac{1}{m-1}\cdot\lambda_{\min}\left( A^{*} A\right)\left\|x_k-x_\star\right\|^2_{2}  \notag\\
&=~\left(1-\frac{1}{|\mathcal{R}_{k}|}\cdot\frac{1}{\sum\limits_{i_k\in \mathcal{R}_{k}}\|A^{(i_k)}\|^2_2 }\cdot\frac{1}{m-1} \cdot\lambda_{\min}\left( A^{*} A\right) \right)\| x_k-x_\star\|^2_{2}.
\end{align}
So the estimate (\ref{5}) is obtained.
%
By induction on the iteration index $k$, we can get the estimate (\ref{6}).
\end{proof}

\begin{remark}
\label{rmk1}
Since $1\leq \alpha \leq m$ and $\min\limits_{1\leq i\leq m}\|A^{(i)}\|^{2}_2\leq \beta \leq \|A\|^{2}_{F}$, it holds that
$$\left(1-\frac{\lambda_{\min}\left( A^{*} A\right) }{\min\limits_{1\leq i\leq m}\|A^{(i)}\|^{2}_2 \cdot (m-1) } \right) \leq\left(1-\frac{\lambda_{\min}\left( A^{*} A\right) }{\alpha\cdot \beta\cdot(m-1)} \right)\leq\left(1-\frac{\lambda_{\min}\left( A^{*} A\right) }{m\cdot \|A\|^{2}_{F}\cdot(m-1)} \right).$$
Hence, the convergence factor of the GK method is small when the parameters $\alpha$ and $\beta$ are small. So, the smaller size of $|\mathcal{R}_{k}|$ is,  the better convergence factor of the GK method is when $\beta$ is fixed. From the definitions of $\mathcal{U}_{k}$, $\mathcal{V}_{k}$, and $\mathcal{R}_{k}$, we can find that
the size of $|\mathcal{R}_{k}|$ may be smaller than those of $|\mathcal{U}_{k}|$ and $|\mathcal{V}_{k}|$. This is one of the reasons that our algorithm behaves better in computing time.
\end{remark}
\begin{remark}\label{rmk12111}


 If $\alpha=1$ and $ \beta=\min\limits_{1\leq i\leq m}\|A^{(i)}\|^{2}_2$, the right side of (\ref{5}) is smaller than
 $$\left(1-\frac{1}{\min\limits_{1\leq i\leq m}\|A^{(i)}\|^{2}_2 \cdot (m-1)}\lambda_{\min}\left( A^{*} A\right)\right)\left\|x_{k}-x_{\star}\right\|^2_{2}.$$
Since
 \begin{align} \label{1000}
 \min\limits_{1\leq i\leq m}\|A^{(i)}\|^{2}_2 \cdot (m-1)\leq\|A\|^2_F-\min\limits_{1\leq i\leq m}\|A^{(i)}\|^{2}_2<\|A\|^2_F,
 \end{align}
which implies
 $$\frac{1}{\min\limits_{1\leq i\leq m}\|A_{(i)}\|^{2}_2 \cdot (m-1)}>\frac{1}{2} \left(\frac{1}{\|A\|_{F}^{2}-\min \limits_{1 \leq i \leq m}\left\|A_{(i)}\right\|_{2}^{2}}+\frac{1}{\|A\|_{F}^{2}}\right ),$$
we have
\begin{align}\label{2000}
&\left(1-\frac{1}{\min\limits_{1\leq i\leq m}\|A^{(i)}\|^{2}_2 \cdot (m-1)}\lambda_{\min}\left( A^{*} A\right)\right)\left\|x_{k}-x_{\star}\right\|_{2}^{2}\nonumber\\
&<\left(1-\frac{1}{2}\left (\frac{1}{\|A\|_{F}^{2}-\min \limits_{1 \leq i \leq m}\left\|A^{(i)}\right\|_{2}^{2}}+\frac{1}{\|A\|_{F}^{2}} \right)\lambda_{\min}\left( A^{*} A\right)\right)\left\|x_{k}-x_{\star}\right\|_{2}^{2}.
\end{align}

Note that the error estimate in expectation of the GRK method in \cite{Bai2018} is
$$
\mathbb{E}_{k}\left\|x_{k+1}-x_{\star}\right\|_{2}^{2} \leq\left (1- \frac{1}{2} \left(\frac{1}{\|A\|_{F}^{2}-\min \limits_{1 \leq i \leq m}\left\|A^{(i)}\right\|_{2}^{2}}+\frac{1}{\|A\|_{F}^{2}} \right)\lambda_{\min } (A^{*} A) \right  )\left\|x_{k}-x_{\star}\right\|_{2}^{2} , $$
where $ k=1,2, \ldots.$ So the convergence factor of GK method is slightly better for the above case.
\end{remark}

For the RGRK method, its error estimate in expectation given in  \cite{Bai2018r} is
$$
\mathbb{E}_{k}\left\|x_{k+1}-x_{\star}\right\|_{2}^{2} \leq\left (1-  \left(\frac{\theta}{\|A\|_{F}^{2}-\min \limits_{1 \leq i \leq m}\left\|A^{(i)}\right\|_{2}^{2}}+\frac{1-\theta}{\|A\|_{F}^{2}}\right)\lambda_{\min } (A^{*} A)  \right  )\left\|x_{k}-x_{\star}\right\|_{2}^{2} , $$
where $ k=1,2, \ldots .$ The estimate 
attains its minimum at $\theta=1$, which is
$$\left(1- \frac{1}{\|A\|_{F}^{2}-\min \limits_{1 \leq i \leq m}\|A^{(i)} \|_{2}^{2}} \lambda_{\min } (A^{*} A) \right)\left\|x_{k}-x_{\star}\right\|_{2}^{2}.$$
Considering \eqref{1000} and similar to derivation of \eqref{2000}, we can get that when
 $\alpha=1$ and $ \beta=\min\limits_{1\leq i\leq m}\|A^{(i)}\|^{2}_2$, 
the convergence factor of GK method is also slightly better than that of the RGRK method.

\section{Experimental results}\label{sec4}
In this section, we 
compare the GRK, RGRK and GK methods 
with the matrix $A\in C^{m\times n}$ from two sets. One is generated randomly by using the MATLAB function \texttt{randn}, and the other includes some full-rank sparse matrices (e.g., ch7-8-b1, ch8-8-b1, model1, Trec8, Stranke94 and mycielskian5) and some rank-deficient sparse matrices (e.g., flower\_5\_1, relat6, D\_11, Sandi\_sandi, GD01\_c and GD02\_a) originating in different applications from \cite{Davis2011}. They possess certain structures, such as square ($m=n$) (e.g., Stranke94, mycielskian5, GD01\_c and GD02\_a), thin ($m>n$) (e.g., ch7-8-b1, ch8-8-b1, flower\_5\_1 and relat6 ) or fat ($m<n$) (e.g., model1, Trec8, D\_11 and Sandi\_sandi), and some properties, such as symmetric (e.g., Stranke94 and mycielskian5) or nonsymmetric (e.g., GD01\_c and GD02\_a).

We compare the three methods mainly in terms of the iteration numbers (denoted as ``IT'') and the computing time in seconds (denoted as ``CPU''). It should be pointed out here that the IT and CPU listed in our numerical results denote the arithmetical averages of the required iteration numbers and the elapsed CPU times with respect to 50 times repeated runs of the corresponding methods, and 
we always set $\theta=1$ in the RGRK method in our experiments since 
the convergence factor attains its minimum in this case. To give an intuitive compare of the three methods, we also present the iteration number speed-up of GK against GRK, which is defined as
\begin{eqnarray*}
\texttt{IT speed-up\_1}=\frac{\texttt{IT of GRK } }{\texttt{IT of GK } },
\end{eqnarray*}
the iteration number speed-up of GK against RGRK, which is defined as
\begin{eqnarray*}
\texttt{IT speed-up\_2}=\frac{\texttt{IT of RGRK } }{\texttt{IT of GK } },
\end{eqnarray*}
the computing time speed-up of GK against GRK, which is defined as
\begin{eqnarray*}
\texttt{CPU speed-up\_1}=\frac{\texttt{CPU of GRK} }{\texttt{CPU of GK} },
\end{eqnarray*}
and the computing time speed-up of GK against RGRK, which is defined as
\begin{eqnarray*}
\texttt{CPU speed-up\_2}=\frac{\texttt{CPU of RGRK} }{\texttt{CPU of GK} }.
\end{eqnarray*}
In addition, for the sparse matrices from \cite{Davis2011}, we define the density as follows
\begin{eqnarray*}
\texttt{density}=\frac{\texttt{number of nonzero of an $m\times n$ matrix}}{\texttt{mn}},
\end{eqnarray*}
and use \texttt{cond(A)} to represent the Euclidean condition number of the matrix $A$. 

In our specific experiments, the solution vector $x_\star$ is generated randomly by the MATLAB function \texttt{randn} and we set the right-hand side $b=Ax_{\star}$. All the test problems are started from an initial zero vector $x_{0}=0$ and terminated once the \emph{relative solution error} (\texttt{RES}), defined by $$\texttt{RES}=\frac{\left\|x_{k}-x_{\star}\right\|^{2}_2}{\left\|x_{\star}\right\|^{2}_2},$$ satisfies $\texttt{RES}\leq10^{-6}$ or the number of iteration exceeds $ 200,000$.

\begin{table}[tp]
  \centering
  \fontsize{6.5}{8}\selectfont
    \caption{ IT and CPU of GRK, RGRK and GK for m-by-n matrices $A$ with $n=50$ and different $m$.}
    \label{tab1}
    \begin{tabular}{|c|c|c|c|c|c|c|}
 \hline
\multicolumn{2}{|c|}{$m\times n$}&$1000\times 50$&$ 2000\times 50$&$3000\times 50$&$4000\times50$&$5000\times50$\cr
\hline
  \multirow{5}{*}{IT}&GRK           & 88.7600 & 79.3200& 75.4200& 74.1200&72.3000\cr
                     &RGRK          & 67.0000 & 57.0000& 50.0000& 51.0000& 48\cr
                     &GK            & 77.0000 & 64.0000& 58.0000& 54.0000&52\cr \cline{2-7}
                     &speed-up\_1   & 1.1527  & 1.2394 & 1.3003 & 1.3726 &1.3904\cr\cline{2-7}
                     &speed-up\_2   & 0.8701  & 0.8906 & 0.8621 & 0.9444 &0.9231\cr\hline

  \multirow{5}{*}{CPU}&GRK          & 0.0475  & 0.0606 & 0.0681 & 0.1241 &0.1416 \cr
                     &RGRK          & 0.0300  & 0.0353 & 0.0394 & 0.0862 &0.0928 \cr
                     &GK            & 0.0066  & 0.0084 & 0.0094 & 0.0222 &0.0278\cr\cline{2-7}
                     &speed-up\_1   & 7.2381  & 7.1852 & \textbf{7.2667} & 5.5915 &5.0899 \cr\cline{2-7}
                     &speed-up\_2   & \textbf{ 4.5714}  & 4.1852 & 4.2000 & 3.8873 &3.3371 \cr\hline

   \end{tabular}
\end{table}

\begin{table}[tp]
  \centering
  \fontsize{6.5}{8}\selectfont
    \caption{ IT and CPU of GRK, RGRK and GK for m-by-n matrices $A$ with $n=100$ and different $m$.}
    \label{tab2}
    \begin{tabular}{|c|c|c|c|c|c|c|}
 \hline
\multicolumn{2}{|c|}{$m\times n$}&$1000\times 100$&$ 2000\times 100$&$3000\times 100$&$4000\times100$&$5000\times100$\cr
\hline
  \multirow{5}{*}{IT}&GRK           & 205.0400& 167.8400& 157.2600& 152    & 146.8600    \cr
                     &RGRK          & 177     & 129     & 120     & 114    & 110         \cr
                     &GK            & 183     & 137     & 122     & 122    & 113       \cr \cline{2-7}
                     &speed-up\_1   & 1.1204  & 1.2251  & 1.2890  & 1.2459 & 1.2996   \cr\cline{2-7}
                     &speed-up\_2   & 0.9672  & 0.9416  & 0.9836  & 0.9344 & 0.9735   \cr\hline

  \multirow{5}{*}{CPU}&GRK          & 0.0959  & 0.0972  & 0.1163  & 0.2744 & 0.3409  \cr
                     &RGRK          & 0.0844  & 0.0747  & 0.0912  & 0.2172 & 0.2512   \cr
                     &GK            & 0.0187  & 0.0256  & 0.0291  & 0.0663 & 0.0791   \cr\cline{2-7}
                     &speed-up\_1   & 5.1167  & 3.7927  & 4       & 4.1415 & 4.3123  \cr\cline{2-7}
                     &speed-up\_2   & 4.5000  & 2.9146  & 3.1398  & 3.2783 & 3.1779   \cr\hline

   \end{tabular}
\end{table}

\begin{table}[tp]
  \centering
  \fontsize{6.5}{8}\selectfont
    \caption{ IT and CPU of GRK, RGRK and GK for m-by-n matrices $A$ with $n=150$ and different $m$.}
    \label{tab3}
    \begin{tabular}{|c|c|c|c|c|c|c|}
 \hline
\multicolumn{2}{|c|}{$m\times n$}&$1000\times 150$&$ 2000\times 150$&$3000\times 150$&$4000\times150$&$5000\times150$\cr
\hline
  \multirow{5}{*}{IT}&GRK           & 364.6800 & 276.0200 & 249.5800 & 233.6800 & 226.4000  \cr
                     &RGRK          & 318      & 239      & 199      & 189      & 179       \cr
                     &GK            & 321      & 245      & 202      & 192      & 183       \cr \cline{2-7}
                     &speed-up\_1   & 1.1361   & 1.1266   & 1.2355   & 1.2171   & 1.2372    \cr\cline{2-7}
                     &speed-up\_2   & 0.9907   & 0.9755   & 0.9851   & 0.9844   & 0.9781    \cr\hline

  \multirow{5}{*}{CPU}&GRK          & 0.1906   & 0.1734   & 0.2712   & 0.6228   & 0.8194    \cr
                     &RGRK          & 0.1675   & 0.1572   & 0.2081   & 0.5209   & 0.6547    \cr
                     &GK            & 0.0462   & 0.0500   & 0.0737   & 0.1556   & 0.2391    \cr\cline{2-7}
                     &speed-up\_1   & 4.1216   & 3.4687   & 3.6780   & 4.0020   & 3.4275    \cr\cline{2-7}
                     &speed-up\_2   & 3.6216   & 3.1437   & 2.8220   & 3.3474   & 2.7386     \cr\hline

   \end{tabular}
\end{table}

\begin{table}[tp]
  \centering
  \fontsize{6.5}{8}\selectfont
    \caption{ IT and CPU of GRK, RGRK and GK for m-by-n matrices $A$ with $n=200$ and different $m$.}
    \label{tab4}
    \begin{tabular}{|c|c|c|c|c|c|c|}
 \hline
\multicolumn{2}{|c|}{$m\times n$}&$1000\times 200$&$ 2000\times 200$&$3000\times 200$&$4000\times200$&$5000\times200$\cr
\hline
  \multirow{5}{*}{IT}&GRK           & 557.7000 & 398.6800 & 351.0400 & 328.3800 & 312.2400 \cr
                     &RGRK          & 517      & 341      & 294      & 277      & 257 \cr
                     &GK            & 504      & 334      & 294      & 264      & 258 \cr \cline{2-7}
                     &speed-up\_1   & 1.1065   & 1.1937   & 1.1940   & 1.2439   & 1.2102 \cr\cline{2-7}
                     &speed-up\_2   & 1.0258   & 1.0210   & 1        & 1.0492   & 0.9961 \cr\hline

  \multirow{5}{*}{CPU}&GRK          & 0.2706   & 0.2797   & 0.4300   & 1.1750   & 1.3834 \cr
                     &RGRK          & 0.2566   & 0.2425   & 0.4034   & 1.0497   & 1.1747 \cr
                     &GK            & 0.0741   & 0.0791   & 0.1197   & 0.3753   & 0.5031 \cr\cline{2-7}
                     &speed-up\_1   & 3.6540   & 3.5375   & 3.5927   & 3.1307   & 2.7497 \cr\cline{2-7}
                     &speed-up\_2   & 3.4641   & 3.0672   & 3.3708   & 2.7968   & 2.3348 \cr\hline

   \end{tabular}
\end{table}

\begin{table}[tp]
  \centering
  \fontsize{6.5}{8}\selectfont
    \caption{ IT and CPU of GRK, RGRK and GK for m-by-n matrices $A$ with $m=50$ and different $n$.}
    \label{tab5}
    \begin{tabular}{|c|c|c|c|c|c|c|}
 \hline
\multicolumn{2}{|c|}{$m\times n$}&$50\times 1000$&$ 50\times 2000$&$50\times 3000$&$50\times4000$&$50\times5000$\cr
\hline
  \multirow{5}{*}{IT}&GRK           & 127.6600 & 114.3800 & 100.9400 & 97.7000 & 95.2000 \cr
                     &RGRK          & 127      & 118      & 99       & 91      & 91  \cr
                     &GK            & 126      & 117      & 98       & 92      & 91  \cr \cline{2-7}
                     &speed-up\_1   & 1.0132   & 0.9776   & 1.0300   & 1.0620  & 1.0462  \cr\cline{2-7}
                     &speed-up\_2   & 1.0079   & 1.0085   & 1.0102   & 0.9891  & 1   \cr\hline

  \multirow{5}{*}{CPU}&GRK          & 0.0594   & 0.0669   & 0.0650   & 0.1247  & 0.1563  \cr
                     &RGRK          & 0.0581   & 0.0625   & 0.0638   & 0.1166  & 0.1412  \cr
                     &GK            & 0.0172   & 0.0275   & 0.0313   & 0.0625  & 0.0766  \cr\cline{2-7}
                     &speed-up\_1   & 3.4545   & 2.4318   & 2.0800   & 1.9950  & 2.0408  \cr\cline{2-7}
                     &speed-up\_2   & 3.3818   & 2.2727   & 2.0400   & 1.8650  & 1.8449 \cr\hline

   \end{tabular}
\end{table}

\begin{table}[tp]
  \centering
  \fontsize{6.5}{8}\selectfont
    \caption{ IT and CPU of GRK, RGRK and GK for m-by-n matrices $A$ with $m=100$ and different $n$.}
    \label{tab6}
    \begin{tabular}{|c|c|c|c|c|c|c|}
 \hline
\multicolumn{2}{|c|}{$m\times n$}&$100\times 1000$&$ 100\times 2000$&$100\times 3000$&$100\times4000$&$100\times5000$\cr
\hline
  \multirow{5}{*}{IT}&GRK           & 285.1800 & 264.6200 & 232.9000 & 217.3200 & 212.3800  \cr
                     &RGRK          & 276      & 255      & 232      & 215      &  208\cr
                     &GK            & 268      & 256      & 226      & 214      & 208\cr \cline{2-7}
                     &speed-up\_1   & 1.0641   & 1.0337   & 1.0305   & 1.0155   & 1.0211\cr\cline{2-7}
                     &speed-up\_2   & 1.0299   & 0.9961   & 1.0265   & 1.0047   & 1\cr\hline

  \multirow{5}{*}{CPU}&GRK          & 0.1412   & 0.1638   & 0.2197   & 0.4278   & 0.5150 \cr
                     &RGRK          & 0.1375   & 0.1497   & 0.2172   & 0.4253   & 0.5031\cr
                     &GK            & 0.0431   & 0.0622   & 0.0788   & 0.1747   & 0.2122\cr\cline{2-7}
                     &speed-up\_1   & 3.2754   & 2.6332   & 2.7897   & 2.4490   & 2.4271\cr\cline{2-7}
                     &speed-up\_2   & 3.1884   & 2.4070   & 2.7579   & 2.4347   & 2.3711\cr\hline

   \end{tabular}
\end{table}

\begin{table}[tp]
  \centering
  \fontsize{6.5}{8}\selectfont
    \caption{ IT and CPU of GRK, RGRK and GK for m-by-n matrices $A$ with $m=150$ and different $n$.}
    \label{tab7}
    \begin{tabular}{|c|c|c|c|c|c|c|}
 \hline
\multicolumn{2}{|c|}{$m\times n$}&$150\times 1000$&$ 150\times 2000$&$150\times 3000$&$150\times4000$&$150\times5000$\cr
\hline
  \multirow{5}{*}{IT}&GRK           & 589.7600 & 441.2800 & 364.5600 & 342.4400 & 340.4400 \cr
                     &RGRK          & 580      & 432      & 355      & 331      & 330 \cr
                     &GK            & 586      & 427      & 358      & 338      & 323 \cr \cline{2-7}
                     &speed-up\_1   & 1.0064   & 1.0334   & 1.0183   & 1.0131   & 1.0540  \cr\cline{2-7}
                     &speed-up\_2   & 0.9898   & 1.0117   & 0.9916   & 0.9793   & 1.0217  \cr\hline

  \multirow{5}{*}{CPU}&GRK          & 0.3700   & 0.3922   & 0.4500   & 0.9569   & 1.2394 \cr
                     &RGRK          & 0.3531   & 0.3503   & 0.4416   & 0.9291   & 1.1884 \cr
                     &GK            & 0.0975   & 0.1291   & 0.2250   & 0.5281   & 0.7097 \cr\cline{2-7}
                     &speed-up\_1   & 3.7949   & 3.0387   & 2        & 1.8118   & 1.7464 \cr\cline{2-7}
                     &speed-up\_2   & 3.6218   & 2.7143   & 1.9625   & 1.7592   & 1.6746 \cr\hline

   \end{tabular}
\end{table}

\begin{table}[tp]
  \centering
  \fontsize{6.5}{8}\selectfont
    \caption{ IT and CPU of GRK, RGRK and GK for m-by-n matrices $A$ with $m=200$ and different $n$.}
    \label{tab8}
    \begin{tabular}{|c|c|c|c|c|c|c|}
 \hline
\multicolumn{2}{|c|}{$m\times n$}&$200\times 1000$&$ 200\times 2000$&$200\times 3000$&$200\times4000$&$200\times5000$\cr
\hline
  \multirow{5}{*}{IT}&GRK           &946.4600 & 641.0400 & 540.6600&497.5000& 471.0200  \cr
                     &RGRK          &932      & 636      & 514     &492     & 455   \cr
                     &GK            &962      & 628      & 521     &499     & 455   \cr \cline{2-7}
                     &speed-up\_1   &0.9838   & 1.0208   & 1.0377  &0.9970  & 1.0352   \cr\cline{2-7}
                     &speed-up\_2   &0.9688   & 1.0127   & 0.9866  &0.9860  & 1   \cr\hline

  \multirow{5}{*}{CPU}&GRK          &0.6241   & 0.6269   & 0.7834  &1.7025  & 2.1916    \cr
                     &RGRK          &0.6238   & 0.5909   & 0.7381  &1.6322  & 2.0866    \cr
                     &GK            &0.1766   & 0.2272   & 0.3756  &1.0044  & 1.2925   \cr\cline{2-7}
                     &speed-up\_1   &3.5345   & 2.7593   & 2.0857  &\textbf{1.6951}  & 1.6956   \cr\cline{2-7}
                     &speed-up\_2   &3.5327   & 2.6011   & 1.9651  &1.6251  & \textbf{1.6144 }   \cr\hline

   \end{tabular}
\end{table}

For the first class of matrices, that is, the randomly generated matrices, 
the numerical results on IT and CPU are listed in \Cref{tab1,tab2,tab3,tab4} when $m>n$, and in \Cref{tab5,tab6,tab7,tab8} when $m<n$. From \Cref{tab1,tab2,tab3,tab4,tab5,tab6,tab7,tab8}, we see that the GK method requires almost the same number of iterations as those of the GRK and RGRK methods but the GK method is more efficient in term of the computing time. The computing time speed-up of GK against GRK is at least 1.6951 (see Table \ref{tab8} for the $200\times 4000$ matrix) and at most 7.2667 (see Table \ref{tab1} for the $3000\times  50$ matrix), and the computing time speed-up of GK against RGRK is at least 1.6144 (see Table \ref{tab8} for the $200\times 5000$ matrix) and at most 4.5714 (see Table \ref{tab1} for the $1000\times  50$ matrix).

\begin{table}[!htbp]\centering
\begin{small}\scriptsize
\caption{IT and CPU of GRK, RGRK and GK for m-by-n matrices $A$ with different $m$ and $n$.}\centering \label{tab9}
 \begin{tabular}{ccccccccccccccccccc}
 \hline
\textbf{name}  &&\textbf{ch7-8-b1}     & \textbf{ch8-8-b1} &\textbf{model1}&\textbf{ Trec8}&\textbf{Stranke94} &\textbf{mycielskian5}   \cr
\hline
 $m \times n$  && $1176 \times 56 $ & $1568 \times 64$ & $362 \times 798$  &$23 \times 84$    &$10 \times 10 $        &$ 23 \times  23 $    \cr

full rank       && Yes                 &    Yes               & Yes          & Yes            &Yes                   &Yes           \cr

density        && 3.57\%              &    3.13\%            & 1.05\%        & 28.42\%            &90.00\%            &26.84\%            \cr

\texttt{cond(A)} &&  4.79e+14         &   3.48e+14           & 17.57        &  26.89              &51.73                &  27.64               \cr
 \hline
\multirow{5}{*}{IT}  &GRK & 103.9800 & 113.0800 & 4.7484e+03 & 1.7655e+03 & 5.6706e+03 &4.2651e+03    \cr
                     &RGRK& 87.0600  & 89       & 4504       & 1646       & 4158       &4268    \cr
                     &GK  & 87       & 89       & 4183       & 1681       & 3636       &4169    \cr\cline{2-8}
 &  speed-up\_1           & 1.1952   & 1.2706   & 1.1352     & 1.0502     & 1.5596     &1.0230    \cr \cline{2-8}
 &  speed-up\_2           & 1.0007   & 1        & 1.0767     & 0.9792     & 1.1436     &1.0237    \cr
\hline
\multirow{5}{*}{CPU} &GRK & 0.0178   & 0.0200   & 0.5381     & 0.1619     & 0.4847     &0.3816    \cr
                     &RGRK& 0.0116   & 0.0119   & 0.5128     & 0.1500     & 0.3497     &0.3700    \cr
                     &GK  & 0.0047   & 0.0047   & 0.1953     & 0.0441     & 0.0800     &0.0872    \cr\cline{2-8}
 &  speed-up\_1           & 3.8000   & 4.2667   & 2.7552     & 3.6738     & \textbf{6.0586}     &4.3763    \cr\cline{2-8}
 &  speed-up\_2           & 2.4667   & 2.5333   & 2.6256     & 3.4043     & \textbf{4.3711}     &4.2437    \cr
\hline
\end{tabular}
\end{small}
\end{table}

\begin{table}[!htbp]\centering
\begin{small}\scriptsize
\caption{IT and CPU of GRK, RGRK and GK for m-by-n matrices $A$ with different $m$ and $n$.}\centering \label{tab10}
 \begin{tabular}{ccccccccccccccccccc}
 \hline
\textbf{name}  &&\textbf{flower\_5\_1}     & \textbf{relat6} &\textbf{D\_11}&\textbf{Sandi\_sandi}&\textbf{GD01\_c} &\textbf{GD02\_a} \cr
\hline
 $m \times n$  && $ 211 \times 201  $ & $ 2340 \times  157 $ & $ 169 \times 461 $  &$ 314 \times 360 $  &$ 33\times33  $    &$  23 \times 23  $ \cr

full rank       && No                 &    No               & No                & No                  &No                  &No      \cr

density        && 1.42 \%              &   2.21  \%        & 3.79 \%            & 0.54 \%             & 12.40\%            &16.45 \%      \cr
\texttt{cond(A)} &&  2.00e+16          &    Inf            &  2.21e+17         &  1.47e+17            &  Inf               & Inf           \cr
 \hline
\multirow{5}{*}{IT}   &GRK& 9.8127e+03 &  1.6099e+03  & 682.8200 &1756         &1.9329e+03 & 1.3928e+03 \cr
                     &RGRK& 9.8616e+03 &  1.5199e+03  & 668      &1.6864e+03   &1819       & 1469 \cr
                       &GK& 10521      &  1510        & 690      &1787         &1823       & 1228  \cr\cline{2-8}
 &  speed-up\_1           & 0.9327     &  1.0661      & 0.9896   &0.9827       &1.0603     & 1.1342  \cr \cline{2-8}
 & speed-up\_2            & 0.9373     &  1.0065      & 0.9681   &0.9437       &0.9978     & 1.1963   \cr
\hline
\multirow{5}{*}{CPU} &GRK&  1.0197     &  0.2550      & 0.0866   &0.1812       &0.1663     & 0.1241   \cr
                    &RGRK&  0.9653     &  0.2353      & 0.0853   &0.1741       &0.1538     & 0.1219    \cr
                      &GK&  0.3006     &  0.0766      & 0.0369   &0.0600       &0.0378     & 0.0303    \cr\cline{2-8}
 &  speed-up\_1          &  3.3919     &  3.3306      &\textbf{ 2.3475 }  &3.0208       &4.3967     & 4.0928     \cr\cline{2-8}
 & speed-up\_2           &  3.2110     &  3.0735      & \textbf{2.3136}   &2.9010       &4.0661     & 4.0206     \cr
\hline
\end{tabular}
\end{small}
\end{table}

For the second class of matrices, that is, the sparse matrices from \cite{Davis2011}, the numerical results on IT and CPU are listed in \Cref{tab9} when the matrices 
are full-rank with different $m$ and $n$, and in \Cref{tab10} when the matrices 
are rank-deficient with different $m$ and $n$. In both tables, the iteration numbers of the GRK, RGRK and GK methods are almost the same, but again the CPUs of the GK method are smaller than those of the other two methods, with the CPU speed-up of GK against GRK being at least 2.3475 (the matrix D\_11 in \Cref{tab10}) and at most 6.0586 (the matrix Stranke94 in \Cref{tab9}), and the CPU speed-up of GK against RGRK being at least 2.3136 (the matrix D\_11 in \Cref{tab10}) and at most 4.3711 (the matrix Stranke94 in \Cref{tab9}).

Therefore, in all the cases, although the GK method requires almost the same number of iterations as those of the GRK and RGRK methods, our method 
outperforms the others in term of the computing time, which is consistent with the analysis before Algorithm \ref{alg3}.

\clearpage

\bibliographystyle{siamplain}
\bibliography{references}
\end{document}